\documentclass[12pt,a4paper]{amsart}

\usepackage{amsmath}
\usepackage{amssymb}

\theoremstyle{plain}   %% This is the default, anyway
\begingroup % Confine the \theorembodyfont command
\newtheorem{theorem}{Theorem}[section]   % Numbered within each section
\newtheorem{lemma}[theorem]{Lemma}         % Numbered along with thm
\endgroup

\theoremstyle{definition}
\theoremstyle{remark}
\newtheorem{notation}[theorem]{Notation}
\newcommand{\vf}{\varphi}
\newcommand{\R}{{\mathbb R}}
\newcommand{\N}{{\mathbb N}}
\newcommand{\Z}{\mathbb Z}

\newcommand{\dist}{\operatorname{dist}}

\usepackage[dvips]{color}

\begin{document}

\title{On the set of points at which an increasing continuous singular function has a nonzero finite derivative   }

\thanks{}

\author{Marta Kossaczk{\'a}}
\author{Lud\v ek Zaj\'\i\v{c}ek}



\email{kossaczka@karlin.mff.cuni.cz}
\email{zajicek@karlin.mff.cuni.cz}

\address{Charles University,
Faculty of Mathematics and Physics,
Sokolovsk\'a 83,
186 75 Praha 8-Karl\'\i n,
Czech Republic}

%h\date{\today}

\begin{abstract} 
S{\'a}nchez, Viader, Parad{\'\i}s and Carrillo (2016) proved that there exists
 an increasing continuous singular function $f$ on $[0,1]$    such that the set $A_f$ of points where $f$ has a nonzero finite derivative has Hausdorff dimension 1 in each subinterval of $[0,1]$. We prove a stronger
 (and optimal) result showing that a set $A_f$ as above can contain any prescribed $F_{\sigma}$ null subset of
 $[0,1]$.
\end{abstract}

\maketitle

\markboth{M. Kossaczk{\'a}, L.~Zaj\'{\i}\v{c}ek}{A note on the set of points at which an increasing singular function has a nonzero finite derivative   }

\section{Introduction}

By a singular function we mean (following \cite{sanchez2012singular}) a continuous nonconstant function with zero derivative almost everywhere.  Motivated by the observation that well-known singular functions have at no point a nonzero finite derivative, the authors 
of \cite{sanchez2012singular} constructed
 an  increasing singular function $f$ on $[0,1]$ such that the set   $A_f$ of points where $f$ has a nonzero finite derivative  is uncountable. Later it was shown
in \cite{sanchez2014singular} that such $A_f$ can be dense in $[0,1]$ and in  
\cite{sanchez2016singular} that $A_f$ can have Hausdorff dimension 1 in each subinterval of $[0,1]$. 

We prove in the present note a stronger
 result showing that a set $A_f$ as above can contain any prescribed $F_{\sigma}$ null subset of
 $[0,1]$,  and that this result on the size of sets $A_f$ is optimal. 
More presisely, we prove the following result.

\begin{theorem}\label{th}
Let $A \subset [0,1]$. Then the following conditions are equivalent.
\begin{enumerate}
\item [(i)] There exits a strictly increasing continuous singular function $f$ on $[0,1]$
  such that  $f$ has a nonzero finite derivative at each point of $A$.
\item [(ii)] $A$ is a subset of an $F_{\sigma}$ Lebesgue null set. 
\end{enumerate}
\end{theorem}

\section{Proof}

 Before the proof of Theorem \ref{th}, we will introduce some notation and  prove two lemmas.
\begin{notation}
The symbol $\lambda$ stands for the Lebesgue measure on $\R $.
 Recall that, by our definition, each singular function is continuous and nonconstant. 
For a function defined on an interval $I$, its limits and derivatives are computed with respect to $I$.
\end{notation}
\begin{lemma}\label{lemacant}
 Let $J=[a_1,a_2]$ be a closed interval. Let $H\subset J$ be an $F_{\sigma}$ set with $\lambda(H)=0$. 
 Then there exists a nondecreasing singular function $\vf$ on $J$ such that
\begin{enumerate}
  \item [(i)] $\vf(a_i)=a_i$, $i=1,2$;
  \item  [(ii)]  $\vf'(x)=0$, $x\in H$.
 \end{enumerate}
\end{lemma}
\begin{proof}
 Since $J \setminus H$ is $G_{\delta}$ and dense in $[0,1]$, its intersection  with  a null dense $G_{\delta}$ set is 
 a null $G_{\delta}$ set  $S \subset J \setminus H$ which is uncountable by Baire category theorem. Since $S$ is uncountable and $G_{\delta}$,
 \cite[Theorem 13.6]{kechris1995classical} implies that $S$ contains a homeomorphic copy $C$ 
 of the Cantor set.
 We can choose 
(e.g. by \cite[Corollary 2.8.]{hebert1968supports})  a nonatomic  Borel finite measure $\nu$ on $C$ whose support  is   $C$. 
Now it is easy to see that the function $\vf$ defined by
$$\vf(x)=\frac{(a_2-a_1)}{\nu(C)} \nu(C \cap [a_1,x])+a_1, \quad x\in J,$$ is nondecreasing singular and satisfies conditions (i) and (ii).
\end{proof}

We will say that $P \subset (a,b)$ is an ``infinite partition of $(a,b)$'' if $a$ and $b$ are the only
 accumulation points of $P$. Clearly, $P$ has this property if and only if there exists a sequence
 $(p_z)_{z \in \Z}$ such that $P = \{p_z:\ z \in \Z\}$, $p_z < p_{z+1},\ z \in \Z,$ $\inf \{p_z:\ z\in \Z\}= a$
 and  $\sup \{p_z:\ z\in \Z\}= b$. Any such sequence $(p_z)_{z \in \Z}$ will be called  an ``ordering of 
$P$''.

 We will show that, for each $(a,b)$, there exists an infinite partition $P$ of $(a,b)$ such that, for each
ordering $(p_z)_{z \in \Z}$ of $P$ and  
  each $z \in \Z$,
\begin{equation}\label{inpa}
 p_{z+1}- p_z < \min \{|p_z-a|^2, |p_{z+1} -b|^2\}. 
\end{equation}
To this end, choose an arbitrary infinite partition $P^*$ of $(0,1)$, its ordering $(p^*_z)_{z \in \Z}$ and, for each $z \in \Z$, choose
 a finite partition $p_z^*=t_{z,0} <  t_{z,1}<\dots< t_{z,k_z}= p_{z+1}^*$ of $[p_z^*,p_{z+1}^*]$ such that
 $\max\{t_{z,i} - t_{z,i-1}:\  1\leq i \leq k_z\} < \min((p_z^*-a)^2, (b-p^*_{z+1})^2)$. Now it is 
 easy to see that $P:= \{t_{z,i}:\  z \in \Z,\ 0\leq i \leq k_z\}$ has the desired property.

\begin{lemma}\label{lema1}
 Let $\{0,1\} \subset F\subset [0,1]$ be a closed set with $\lambda(F)=0$. Let $M\subset (0,1)\setminus F$ be an $F_{\sigma}$ set with $\lambda(M)=0$.
 Then there exists a nondecreasing singular function $g$ on $[0,1]$ such that
  \begin{enumerate}
  \item [(i)] $g(z)=z$,\ $z \in F$;
  \item [(ii)] $g'(z)=0, \ z\in M$;
  \item [(iii)] $|g(x) -x| \leq \dist^2(x,F),\ \ x \in [0,1].$
 \end{enumerate}
\end{lemma}
\begin{proof}
 Set $G=(0,1)\setminus F$. Then $G=\bigcup_{\alpha \in A} I_{\alpha} $,  where $(I_{\alpha})_{\alpha \in A}$
 is a (nonempty countable) disjoint system of open intervals.
For each $\alpha \in A$, choose an infinite partition $P_{\alpha}$ of $I_{\alpha}=: (a_{\alpha}, b_{\alpha})$
 and its ordering $(p_z^{\alpha})_{z \in \Z}$ such that \eqref{inpa} holds for  $[a,b]:= [a_{\alpha}, b_{\alpha}]$
 and $p_z:= p_z^{\alpha}$. For any  $\alpha \in A$ and $z \in \Z$ we apply Lemma \ref{lemacant}
 to $[a_1,a_2]:= [p_z^{\alpha}, p_{z+1}^{\alpha}]$ and $H:= (M \cap [p_z^{\alpha}, p_{z+1}^{\alpha}])\cup
 \{p_z^{\alpha}, p_{z+1}^{\alpha}\}$ and obtain a nondecreasing  singular function $\vf_z^{\alpha}$
 on $[p_z^{\alpha}, p_{z+1}^{\alpha}]$ such that 
\begin{equation}\label{hv}
\vf_z^{\alpha}(p_z^{\alpha}) = p_z^{\alpha},\ \vf_z^{\alpha}( p_{z+1}^{\alpha}) =  p_{z+1}^{\alpha},\ 
\end{equation}
\begin{equation}\label{dvhv}
  (\vf_z^{\alpha})'_+ (p_z^{\alpha})  =  (\vf_z^{\alpha})'_- ( p_{z+1}^{\alpha})=0\ \ \      
\text{and}\ \ \  (\vf_z^{\alpha})'(x)=0,\ x \in M \cap (p_z^{\alpha}, p_{z+1}^{\alpha}).
 \end{equation}
For each  $\alpha \in A$, put  $g_{\alpha}(x):= \vf_z^{\alpha}(x),\ x \in  [p_z^{\alpha}, p_{z+1}^{\alpha}]$.
 By \eqref{hv}, the definition is correct and $g_{\alpha}$ is a continuous nondecreasing function on
 $I_{\alpha}$ which has a.e. zero derivative. Moreover, \eqref{dvhv} implies that  $g_{\alpha}'(x)=0$ for
 each $x \in M \cap I_{\alpha}$. 

Finally put  $g(x):= x$ if $x \in F$ and $g(x):= g_{\alpha}(x)$ if $x \in I_{\alpha}$. Then $g$ is clearly
 nondecreasing on all $[0,1]$, has a.e. zero derivative, is continuous on each $I_{\alpha}$ and properties
 (i) and (ii) hold. 

The inequality of (iii) is trivial for  $x\in F$. If $x \in [0,1] \setminus F$, then $x \in I_{\alpha}= (a_{\alpha}, b_{\alpha})$ for some $\alpha \in A$
 and we can choose $z\in \Z$ such that $x \in [p_z^{\alpha}, p_{z+1}^{\alpha}]$. Then we have
 $g(x)=  g_{\alpha}(x) = \vf_z^{\alpha}(x)$ and, since $\vf_z^{\alpha}(x) \in [p_z^{\alpha}, p_{z+1}^{\alpha}]$,
 we obtain
$$ |g(x)-x| = |\vf_z^{\alpha}(x) - x| \leq  p_{z+1}^{\alpha}- p_z^{\alpha} \leq \min\{|p_z^{\alpha}-a_{\alpha}|^2,
 |p_{z+1}^{\alpha}-b_{\alpha}|^2\} \leq \dist^2(x,F)$$
  and so (iii) is proved.
	
	It remains to prove that $g$ is continuous on all $[0,1]$. We know that $g$ is continuous on $G$ and
	 (iii) implies that, if $a \in F$ is fixed, we have for each  $x \in [0,1]$
	$$ |g(x)-g(a)| = |g(x)-a|\leq |g(x)-x| + |x-a| \leq |x-a|^2 + |x-a| \to 0,\ \ x \to a.$$
\end{proof}

 {\bf Proof of Theorem.}
\smallskip

To prove the implication ``$(i) \Rightarrow (ii) $'', suppose that $f$ is a singular increasing function on $[0,1]$
 and $0< f'(x) < \infty$ for each $x \in A$. Let  $f^*$ be a continuous extension of $f$ to $\R$
 and denote  $E_f:= \{x\in (0,1):\ -\infty < f'(x)<\infty\}$. Then $f'(x)= \lim_{n \to \infty}
 n (f^*(x+1/n) - f^*(x)),\ x \in E_f$,\ and so $ f'|_{E_f}$ is a first Baire class function on $E_f$.
 Consequently $P:=\{x\in E_f: f'(x) >0\}$ is an $F_{\sigma}$ subset in the space $E_f$ 
(see e.g. \cite [Theorem 10.12]{bruckner1997real})
 and consequently there exists an $F_{\sigma}$ set $H \subset [0,1]$ such that
 $P= E_f \cap H$. Since $f$ is singular, $P$ is a null set,
 $[0,1]\setminus E_f$ is a null set and consequently also $H \subset P \cup ([0,1]\setminus E_f)$ is null.
 Therefore $A \subset P \cup \{0,1\}$ is contained in an $F_{\sigma}$ null set $H \cup \{0,1\}$.
 \smallskip

To prove the  implication ``$(ii) \Rightarrow$  (i)'', let $A \subset M$, where $M \subset [0,1]$ is an $F_{\sigma}$ null set. 
Write $M=\bigcup_{n=1}^{\infty} F_n$, where $F_n$ are closed subsets of $[0,1]$.
  Without any loss of generality we can assume that $M$ is dense in $[0,1]$,
  $ F_1\subset F_2\subset F_3 \subset \dots$ and $\{0,1\}\subset F_1$. 
  
 Set $M_n: = M\setminus F_{n}$, $n \in \N$. Then $M_n$ is an $F_{\sigma}$ set with $\lambda(M_n)=0$ and $M_n\subset (0,1)\setminus F_n$.
 Applying  Lemma \ref{lema1} to the sets $F_n$ and $M_n$ we obtain nondecreasing singular functions 
$g_n$ on $[0,1]$ such that, for each $n \in \N$,
\begin{equation}\label{osn} 
 \text{the conditions (i) - (iii) of  Lemma \ref{lema1} hold for}\ \ g:=g_n,\   F:= F_n,\ \ M:=M_n.
\end{equation}
 Now set
 $$
  f(x)=\sum_{n=1}^{\infty}\frac{1}{2^n}g_n(x), \quad x\in [0,1].
 $$
  Clearly $f$ is continuous and nondecreasing on $[0,1]$.
According to the Fubini theorem on derivative of a sum of monotone functions 
\cite[Theorem 17.18]{hewitt1969real} the function $f$ has zero derivative almost everywhere.

Now we will  prove that 
\begin{equation}\label{nzf}
\text{at each point $a\in M$, there exists a nonzero finite $f'(a)$.}
\end{equation}
So fix an arbitrary $a\in M$. Let $n\in \N$ be the natural number such that $a\in F_n\setminus F_{n-1}$, (where 
$F_0: =\emptyset$).
If $1\leq k <n$, then  $a\in M_k$  and so \eqref{osn} implies
$$
\Bigl(\frac{1}{2^k}g_k\Bigr)'(a)=0.
 $$
		Consequently, to prove  $0< f'(a) < \infty$,  it is sufficient to show that
$$  r_n'(a)= 	\sum_{k=n}^{\infty}\frac{1}{2^k},\ \ \text{where}\ \ 
r_n=\sum_{k=n}^{\infty}\frac{1}{2^k} g_k.$$
And this equality is true since, using  \eqref{osn} and $a\in F_k$ for $k\geq n$, we obtain  
\begin{multline*}
|r_n(x) -r_n(a) - \sum_{k=n}^{\infty}\frac{1}{2^k}(x-a)| \leq \sum_{k=n}^{\infty}\frac{1}{2^k}|g_k(x) - g_k(a) - (x-a)|\\
= \sum_{k=n}^{\infty}\frac{1}{2^k} |g_k(x) -x| \leq \sum_{k=n}^{\infty}\frac{1}{2^k} \dist^2(x,F_k)
 \leq |x-a|^2 = o(|x-a|), \ \ x\to a.
\end{multline*} 
Observe that $f$ is strictly increasing since	it is nondecreasing, $M$ is dense and 
\eqref{nzf} holds.	Using 	\eqref{nzf} and $A\subset M$, we conclude that (i) holds, since $f$ has all desired properties.

\end{document}